\documentclass[a4paper,leqno,11pt,twoside]{amsart}
\usepackage{amsmath,amsfonts,amssymb,amsthm,amscd}
\usepackage[utf8]{inputenc}
\usepackage[english]{babel}
\usepackage[colorlinks=true,citecolor=blue, urlcolor=blue, linkcolor=blue,pagebackref]{hyperref}
\usepackage[top=1.2in,bottom=1.2in,left=1.in,right=1.in]{geometry} 
\usepackage{tikz}
\usetikzlibrary{matrix}
\usetikzlibrary{arrows}
\usepackage{xfrac}
\usepackage{tikz-cd}
\usepackage{graphicx}
\usepackage{pgf,tikz}
\usetikzlibrary{arrows}

\usepackage[all]{xy} 

\usepackage{enumerate}
\usepackage{xcolor}
\usepackage{aliascnt}

\theoremstyle{plain}
\newtheorem{thm}{Theorem}[section]

\newtheorem{thmIntr}{Theorem}

\newaliascnt{corIntr}{thmIntr}

\newaliascnt{lem}{thm}
\newtheorem{lem}[lem]{Lemma}
\aliascntresetthe{lem}

\newaliascnt{cor}{thm}
\newtheorem{cor}[cor]{Corollary}
\aliascntresetthe{cor}

\newaliascnt{prop}{thm}
\newtheorem{prop}[prop]{Proposition}
\aliascntresetthe{prop}

\newaliascnt{quest}{thm}

\aliascntresetthe{quest}

\theoremstyle{definition}

\newaliascnt{rem}{thm}
\newtheorem{rem}[rem]{Remark}
\aliascntresetthe{rem}

\newaliascnt{defn}{thm}
\newtheorem{defn}[defn]{Definition}
\aliascntresetthe{defn}

\newaliascnt{ex}{thm}

\aliascntresetthe{ex}

\newtheorem*{claim}{Claim}

\numberwithin{equation}{section}

\def\bP{\ensuremath{\mathbb{P}}}
\def\bQ{\ensuremath{\mathbb{Q}}}

\def\bC{\ensuremath{\mathbb{C}}}

\def\cA{\ensuremath{\mathcal{A}}}
\def\cZ{\ensuremath{\mathcal{B}}}
\def\cC{\ensuremath{\mathcal{C}}}
\def\cD{\ensuremath{\mathcal{D}}}

\def\cJ{\ensuremath{\mathcal{J}}}
\def\cM{\ensuremath{\mathcal{M}}}
\def\cN{\ensuremath{\mathcal{N}}}
\def\cO{\ensuremath{\mathcal{O}}}
\def\cP{\ensuremath{\mathcal{P}}}

\def\cR{\ensuremath{\mathcal{R}}}
\def\cZ{\ensuremath{\mathcal{S}}}
\def\cT{\ensuremath{\mathcal{T}}}

\def\cV{\ensuremath{\mathcal{V}}}
\def\cW{\ensuremath{\mathcal{W}}}

\def\cZ{\ensuremath{\mathcal{Z}}}

\newcommand{\tC}{{\widetilde C}}

\newcommand{\tT}{{\widetilde T}}
\def\tcJ{\ensuremath{\widetilde{\mathcal{J}}}}
\def\tcR{\ensuremath{\widetilde{\mathcal{R}}}}
\def\tcP{\ensuremath{\widetilde{\mathcal{P}}}}

\def\hcZ{\ensuremath{\widehat{\mathcal{Z}}}}

\def\ocM{\ensuremath{\overline{\cM}}}
\def\ocR{\ensuremath{\overline{\mathcal{R}}}}
\def\ocP{\ensuremath{\overline{\mathcal{P}}}}
\def\ocA{\ensuremath{\overline{\mathcal{A}}}}

\newcommand{\lra}{\longrightarrow}

\DeclareMathOperator{\Nm}{Nm}
\DeclareMathOperator{\Pic}{Pic}

\DeclareMathOperator{\Bl}{Bl}
\DeclareMathOperator{\Sym}{Sym}

\definecolor{applegreen}{rgb}{0.55, 0.71, 0.0}

\makeatletter
\newcommand{\mylabel}[2]{#2\def\@currentlabel{#2}\label{#1}}
\makeatother

\newcommand{\set}[1]{\left\{#1\right\}}
\newcommand\restr[2]{{
 \left.\kern-\nulldelimiterspace 
 #1 
 \vphantom{\big|} 
 \right|_{#2} 
 }}

\title[Monodromy of the Prym map and semicanonical pencils]{Monodromy of the Prym map and semicanonical pencils in genus 6}

\author{Martí Lahoz}
\address{Martí Lahoz \newline 1. Departament de Matem\`atiques i Inform\`atica, Universitat de Barcelona, Gran Via de les Corts Catalanes, 585, 08007 Barcelona, Spain \newline 2. Centre de Recerca Matemàtica, Edifici C, Campus Bellaterra, 08193 Bellaterra, Spain }
\email{marti.lahoz@ub.edu}
\urladdr{\url{http://www.ub.edu/geomap/lahoz/}}

\author{Juan Carlos Naranjo}
\address{Juan Carlos Naranjo \newline 1. Departament de Matem\`atiques i Inform\`atica, Universitat de Barcelona, Gran Via de les Corts Catalanes, 585, 08007 Barcelona, Spain \newline 2. Centre de Recerca Matemàtica, Edifici C, Campus Bellaterra, 08193 Bellaterra, Spain }
 \email{jcnaranjo@ub.edu}
 \urladdr{\url{http://webgrec.ub.edu/webpages/000006/ang/jcnaranjo.ub.edu.html}}

\author{Andrés Rojas}
\address{Andrés Rojas \newline Institut für Mathematik, Humboldt-Universität zu Berlin, Unter den Linden 6, 10099 Berlin, Germany}
\curraddr{Departament de Matem\`atiques i Inform\`atica, Universitat de Barcelona, Gran Via de les Corts Catalanes, 585, 08007 Barcelona, Spain}
\email{andresrojas@ub.edu} 
\urladdr{\url{https://www.ub.edu/geomap/rojas/}}

\author{Irene Spelta}
\address{Irene Spelta \newline 
 Centre de Recerca Matemàtica, Edifici C, Campus Bellaterra, 08193 Bellaterra, Spain }
\curraddr{Institut für Mathematik, Humboldt-Universität zu Berlin, Unter den Linden 6, 10099 Berlin, Germany}
\email{irene.spelta@hu-berlin.de}
\urladdr{\url{https://amor.cms.hu-berlin.de/~speltair/}}

\begin{document}

\begin{abstract}
The Prym map $\cP_6$ in genus 6 is dominant and generically finite of degree 27.
When restricted to the divisor of curves with an odd semicanonical pencil $\mathcal{T}_6^o$, it is still generically finite, but of degree strictly smaller. 
In this paper, we prove that $\cP_6$ restricted to $\mathcal{T}_6^o$ is birational and that the monodromy group over the image of $\mathcal{T}_6^o$ is the Weyl group $WD_5$.
Thus, there are two other irreducible divisors in the moduli space of Prym curves $\cR_6$ and the degree of $\cP_6$ restricted to them is 10 and 16.
Moreover, we study the geometry of the divisor where $\cP_6$ has degree 10. 
\end{abstract}

\keywords{Prym varieties, theta-characteristics, Prym map, special divisors in the moduli space, monodromy.}
\subjclass[2020]{14H40, 14H45, 14H51, 14K10}

\maketitle

\setcounter{tocdepth}{1}

\section{Introduction}

It is a classical fact, dating back to Wirtinger, that a general principally polarized abelian variety (ppav) of dimension $\leq5$ is a Prym variety. In other words, the \emph{Prym map} $\cP_g:\cR_g\to\cA_{g-1}$ between the moduli spaces $\cR_g$ (of étale double covers of smooth genus $g$ curves) and $\cA_{g-1}$ (of $(g-1)$-dimensional ppavs) is dominant for $g\leq 6$. This allows to understand ppavs of dimension $\leq5$ in terms of geometry of curves, with applications ranging from unirationality of their moduli space (\cite{do_unir, verra}) to the geometry of theta divisors (\cite{izaditheta4, izaditheta5, fgsmv}).

In the case $g=6$, the Prym map is generically finite of degree 27, as proved by Donagi-Smith \cite{donsmith}. Furthermore its monodromy group equals the Weyl group $WE_6\subset \mathfrak{S}_{27}$, and Donagi's tetragonal construction endows its general fiber with the incidence structure of the 27 lines on a smooth cubic surface (\cite{do_tetr,do_fibres}).
It turns out that the rich geometry of the map $\cP_6$ can be used to describe effective divisors on both moduli spaces $\cR_6$ and $\cA_5$, and the monodromy group of $\cP_6$ over such divisors becomes an interesting question.

The purpose of this note is to study this problem in a concrete geometric example. More precisely, let $\cT_g\subset\cM_g$ denote the divisor of (isomorphism classes of) curves $C$ with a \emph{semicanonical pencil}\footnote{In the literature, this is frequently called a \emph{vanishing theta-null}.} (i.e., with a theta-characteristic $L\in\Pic^{g-1}C$ such that $h^0(C,L)$ is even and positive).
The preimage of $\cT_g$ in $\cR_g$ (via the forgetful map $\pi:\cR_g\to\cM_g$) decomposes as the union of two divisors, according to a parity condition. 
We denote these two divisors by $\cT^e_g$ and $\cT^o_g$ (for \emph{even} and \emph{odd} Prym semicanonical pencils, respectively), which are known to be irreducible (\cite{maroj}).

Recent work by the first three authors \cite{lnr} addresses the study of the Prym map restricted to $\cT^e_g$ and $\cT^o_g$. Whereas the case of $\cT^e_g$ is relatively elementary, the analysis for $\cT^o_g$ is substantially harder and is often related to rich geometry. For instance, the study of $\cP_5|_{\cT^o_5}$ reveals enumerative properties of lines on cubic threefolds (\cite[Section 6]{lnr}). 

In the case $g=6$, it is proved that $\cP_6|_{\cT_6^o}$ is generically finite onto its image, of degree strictly less than $27$ (in contrast to $\cP_6|_{\cT^e_6}$). This opens the way to consider other effective divisors on $\cR_6$, namely the irreducible components of $\cP_6^{-1}\,(\overline{\cP_6(\cT_6^o)})$. The degree of $\cP_6$ restricted to each component, as well as the monodromy group of $\cP_6$ over $\overline{\cP_6(\cT_6^o)}$, arise as natural questions in this context.

Our main result is:

\begin{thmIntr}\label{mainthm}
Let $\cZ\subset\cA_5$ denote the divisor obtained as the closure of $\cP_6(\cT^o_6)$. Then the preimage $\cP_6^{-1}(\cZ)$ consists of three irreducible components, namely:

\begin{enumerate}[{\rm (1)}]
 \item $\cT^o_6$, which is birational to $\cZ$: $\deg\left(\cP_6|_{\cT^o_6}\right)=1$.

 \item $\pi^{-1}\cD$, where  
  \[
  \;\;\;\cD=\set{[C]\in\cM_6\;|\;\exists M\in W^1_4(C)\text{ and }x,y,z\in C\text{ with }\omega_C\otimes M^{-1}\cong\cO_C(2x+2y+2z)}
  \]
is the divisor whose general point is a genus 6 curve admitting a plane sextic model with a tritangent line. The equality $\deg\left(\cP_6|_{\pi^{-1}\cD}\right)=10$ holds.

 \item A third component $\mathcal{L}$, which satisfies $\deg\left(\cP_6|_{\mathcal{L}}\right)=16$.

 \end{enumerate}
 
Furthermore, the monodromy group of $\cP_6^{-1}(\cZ)\to\cZ$ equals $WD_5$.
\end{thmIntr}

Let us sketch briefly the argument.
Since $\cZ$ is not the branch divisor of $\cP_6$, the general fiber of $\cP_6^{-1}(\cZ)$ can be identified with the set of 27 lines on a smooth cubic surface: two elements in the fiber are tetragonally related if, and only if, the corresponding lines intersect.

In a first part, we apply the tetragonal construction to a general element $(C,\eta)\in\cT^o_6$; independently of the chosen $g^1_4$ on $C$, one always recovers two elements in $\pi^{-1}\cD$ (with their \emph{tritangent} $g^1_4$).
Conversely, the tetragonal construction applied to a general element in $\pi^{-1}\cD$ (after choosing a tritangent $g^1_4$) returns an element in $\cT^o_6$ and another element in $\pi^{-1}\cD$ (again with a tritangent $g^1_4$).

Since the general curve in $\cD$ has a unique tritangent $g^1_4$, this implies $\deg\left(\cP_6|_{\cT^o_6}\right)=1$, $\pi^{-1}\cD\subset \cP_6^{-1}(\cZ)$ and $\deg\left(\cP_6|_{\pi^{-1}\cD}\right)=10$. In particular, the monodromy group of $\cP_6^{-1}(\cZ)\to\cZ$ must be contained in $WD_5$ (the stabilizer in $WE_6$ of a line).

In a second part, we prove that the monodromy is the entire $WD_5$. We exploit the fact that the Jacobian locus is contained in $\cZ$ to consider an appropriate blown-up map $\widetilde{\cP_6^{-1}(\cZ)}\to\widetilde{\cZ}$, for which we can determine $WD_5$ as its monodromy group. The idea is inspired by \cite[Theorem~4.4]{do_fibres}, but does not follow from the results there; some other arguments, involving degree 4 del Pezzo surfaces and transversality of $\cT^o_6$ with the locus of coverings of trigonal curves, must be employed.

\subsection*{Acknowledgements.} We are grateful to Jieao Song for providing the content of \autoref{althypersurface}. 
We also thank Carlos d'Andrea and Alicia Dickenstein for useful information on polynomial algebra.

The four authors were partially supported 
by the Spanish MINECO research project PID2023-147642NB-I00, 
by the Departament de Recerca i Universitats de la Generalitat de Catalunya (2021 SGR 00697),
and by the Spanish State Research Agency, through the Severo Ochoa and María de Maeztu Program for Centers and Units of Excellence in R\&D (CEX2020-001084-M).
The third author was also supported by the ERC Advanced Grant SYZYGY, funded by the European Research Council.
The fourth author is a member of GNSAGA (INdAM) and has been partially supported by INdAM-GNSAGA project CUP E55F22000270001.

\section{Preliminaries}\label{sec:Prel}
In this section, we review some basics that we will need later. Through the paper, we work over the field of the complex numbers.

\subsection{The Prym map and the divisors of Prym semicanonical pencils}
Consider the moduli space of \emph{smooth Prym curves} of genus $g$
\[
\cR_g=\{(C,\eta) \mid [C]\in \mathcal M_g,\; \eta \in JC[2]\setminus\{\cO_C\}\}/\cong
\]
parametrizing étale, irreducible double covers of smooth genus $g$ curves.
Given $(C,\eta)\in\cR_g$ and $f: \tC\to C$ its associated cover, one can define the \emph{Prym variety} $P:=P(C,\eta)$ as the connected component of the kernel of the induced norm map $\Nm_f:J\tC\to JC$ that contains $0_{J\tC}$. 
The principal polarization on $J\tC$ restricts to twice a principal polarization on $P$ (\cite[Section 3]{mu}), thus we obtain the \emph{Prym map}
\[ \cP_g: \cR_g\to \cA_{g-1}.\]

Let $\cT_g\subset \cM_g$ be the irreducible divisor given by curves admitting a \emph{semicanonical pencil}, i.e. a theta-characteristic $L\in\Pic^{g-1}C$ such that $h^0(C,L)$ is even and positive. In the literature our semicanonical pencils are also called \emph{vanishing theta-nulls}, but later we will use the theta-null divisor in the context of abelian varieties, so we prefer to use this name to avoid any confusion.

By using the forgetful map $\pi:\cR_g \to \cM_g$, we consider the following irreducible divisors:
\[
\begin{aligned}
&\cT^e_g=\set{(C,\eta)\in\cR_g\mid C\text{ has a semicanonical pencil $L$ with $h^0(C,L\otimes\eta)$ even}}\\
&\cT^o_g=\set{(C,\eta)\in\cR_g\mid C\text{ has a semicanonical pencil $L$ with $h^0(C,L\otimes\eta)$ odd}}.
\end{aligned}
\] 
Thanks to the description of the singularities of the theta divisor of a Prym variety provided in \cite[Section 7]{mu}, it is known that $\cP_g$ maps $\cT_g^e$ to the divisor $\theta_{null}\subset \cA_{g-1}$ of principally polarized abelian varieties whose (symmetric) theta divisor contains a singular 2-torsion point of even multiplicity.
More precisely, let \[P^+:= \set{M\in \Pic^{2g-2}\tC\mid \Nm_f(M)=\omega_C \ \text{and } h^0(\tC, M) \text{ is even} }\]
be a presentation of the Prym variety in $\Pic^{2g-2}\tC$.
Then $\Sigma^+= \{M \in P^+\mid h^0(\tC, M) \geq 2\}$ is a ``canonical'' presentation of the theta divisor of $P$.
This shows that an even semicanonical pencil provides an easy example of a singular (exceptional) point in $\Sigma^+$.
The 2-torsion property follows from the fact that 2-torsion points of $P$ in $\Pic^0\tC$ are, in $P^+$, the theta characteristics of $\tC$ lying in $P^+$.

In \cite[Theorems A-B]{lnr}, the first three authors showed that $\cT_g^e=\cP_g^{-1}(\theta_{null})$ for $g\geq 3$ and that $\cT^o_g$ dominates $\cA_{g-1}$ as long as $\dim \cT^o_g\geq \cA_{g-1}$.
Furthermore, they proved that if $g\geq 6$ then the restrictions $\cP_6|_{\cT^e_6}$ and $\cP_6|_{\cT^o_6}$ are generically finite onto their image.
In particular, when $g=6$, the Prym map $\cP_6|_{\cT^e_6}$ has generically degree 27, whereas $\cP_6|_{\cT^o_6}$ has generically degree strictly smaller than 27. 

\subsection{Brill-Noether loci on Prym varieties}
Given an étale double cover $(C, \eta)\in\cR_g$, one can consider a Brill-Noether theory for the Prym variety $P$.
This is due to Welters (\cite{welters}), who defined the \emph{Prym-Brill-Noether loci} as follows:
\[
V^r(C, \eta):=\set{M\in \Pic^{2g-2}\tC\mid \Nm_f(M)=\omega_C, \, h^0(\tC, M)\geq r+1, \, h^0(\tC, M)\equiv r+1 \text{(mod } 2) },
\] 
with the scheme structure defined by $P^+ \cap W^{r}_{2g-2}(\tC) $ when $r$ is odd, and by $P^- \cap W^{r}_{2g-2}(\tC) $ when $r$ is even (here $P^-$ is defined analogously to $P^+$, by requiring an odd number of sections). 

In particular, we have $V^0(C, \eta)=P^-$ and $V^1(C, \eta)=\Sigma^+$. Moreover, being $T(\tC)$ the theta-dual of the Abel-Prym curve $\tC$ inside $P$ (which parametrizes the translates of $\tC\subset P$ contained in the theta divisor), we have $V^2(C, \eta)=T(\tC)$ if $g\geq 4$ and $C$ is not hyperelliptic (\cite{ln}).

If $(C,\eta )\in \cT^o_g$ and $L$ is a semicanonical pencil on $C$ with $h^0(L\otimes\eta)$ odd, then $f^*L\in V^2(C,\eta)$.
Moreover $f^*L$ is a singular point of $V^2(C,\eta)$ (see for instance \cite[Lemma~2.2]{lnr}); if $P^-$ is identified with $P(C,\eta)$ by using a theta-characteristic of $\tC$ lying in $P^-$, then $V^2(C,\eta)$ becomes a symmetric subvariety of $P(C,\eta)$ which is singular at a 2-torsion point.

\subsection{Trigonal and tetragonal constructions, and the Prym map in genus 6}\label{sec:tri_tetr_con}
The classical \emph{trigonal construction} due to Recillas \cite{rec} provides a bijection between étale double covers of genus $g$ trigonal curves and tetragonal curves of genus $g-1$.
More precisely, the Prym variety attached to a cover $\tT\to T$ of a trigonal curve is isomorphic, as a principally polarized abelian variety, to the Jacobian of a curve $X$ equipped with a $g^1_4$ $M$.
Moreover, one can recover $\tT$ (with its natural involution) from $JX$ as the locus of elements $p+q\in X^{(2)}$ such that $h^0(X,M(-p-q))>0$.

Similarly, Donagi \cite{do_tetr} found the so-called \emph{tetragonal construction}, which produces from an \'etale double cover $\tC_1 \to C_1$ of a genus $g$ curve $C_1$ with a fixed $g^1_4$ $M_1$, the data of two more such covers of tetragonal curves $(\tC_i\to C_i, M_i)$, $i=2,3$, such that the three Prym varieties $P(\tC_i\to C_i)$ are isomorphic.
He reformulated in \cite[Lemma 5.5]{do_fibres} this construction in the following way, which is more convenient for our purposes: 

\begin{lem}\label{trigtetr}
There is a bijection between the following two sets of data:
\begin{enumerate}
    \item Triples $(T,N,W)$, where $T\in\cM_{g+1}$ is a trigonal curve, $N$ is a $g^1_3$ on $T$ and $W=\{0,\mu_1,\mu_2,\mu_3\}\subset JT[2]$ is a totally isotropic subgroup with respect to the Weil pairing.
    
    \item A tetragonally related triple $\{(C_i,\eta_i,M_i)\}_{i=1,2,3}$ with $(C_i,\eta_i)\in\cR_g$ and $M_i$ a $g^1_4$ on $C_i$.
\end{enumerate}
\end{lem}
In this bijection, the triplet $(T, N, \mu_i)$ corresponds to $(C_i, M_i)$ via Recillas' construction, whereas $\mu_j\in \langle\mu_i\rangle^{\perp}$ ($j\neq i$) corresponds to $\eta_i\in JC_i[2]$ via Mumford's exact sequence describing the 2-torsion subgroup of a Prym variety (\cite[Section 3]{mu}).

Finally, let us recall the main properties of the Prym map in genus $6$, as discussed in \cite{donsmith} and \cite{do_fibres}:

\begin{thm} \label{prymmapgenus6}
The Prym map $\cP_6: \cR_6 \to \mathcal A_5 $ is dominant and generically finite, of degree $27$. Moreover:
\begin{enumerate}
 
 \item 
 The general fiber of $\cP_6$ can be identified with the set of 27 lines on a smooth cubic surface, and two elements in the fiber are tetragonally related if, and only if, the corresponding lines intersect. 
 \item\label{mongenus6} The monodromy group of $\cR_6$ over $\mathcal A_5$ is the Weyl group $WE_6$, the symmetry group of the incidence of the $27$ lines on a smooth cubic surface.
\end{enumerate}
\end{thm}

\subsection{Compactifications}\label{sec:Compactification}

In order to compute complete fibres of the Prym map, it is convenient to compactify the involved moduli spaces. Abusing notation, $\cT_6^o$ will also denote the closure of this divisor in $\overline {\cR_6}$ (the Deligne-Mumford compactification) or in $\cR_6'$ (Beauville's partial compactification by admissible covers).
We will denote by $\cP_6'\colon \cR_6'\to \cA_5$ Beauville's proper extension of the Prym map (\cite{be_invent}).

Following \cite{fgsmv}, we consider the rational Prym map $\ocP:\ocR_6\dasharrow\ocA_5$ obtained by extending the Prym map to the open subset of $\ocR_6$ lying over the locus of stable curves in $\ocM_6$ with at most one node. 
Here $\ocA_5$ stands for the perfect cone compactification of $\cA_5$, whose rational Picard group $\Pic(\ocA_5)_\bQ$ is generated by the Hodge class $L$ and the class $D$ of the irreducible boundary divisor. In \cite[Proposition~7.3]{lnr}, it was proved that the class of the divisor $\cZ=\overline{\cP_6(\cT^o_6)}$ appearing in \autoref{mainthm} is proportional to $10584L-1320D$.

Arguing as in loc. cit., it is easy to conclude that the fiber of $\cP_6$ at a general point of $\cZ$ is contained in $\cR_6$ and consists of $27$ coverings, all different to each other.

\section{Projective geometry of trigonal curves of genus 6}\label{sec:divisorD}

In this section, we present some results on genus $6$ curves, all of them related with trigonal curves. 
First, we consider the locus $\cD$ in $\cM_6$ of curves which can be represented as plane sextics (with four nodes) with a tritangent line. We will easily prove that the trigonal locus $\cM_6^{trig}\subset\cM_6$ is contained in $\cD$. As we will see in the next section, these sextics appear naturally when we apply the tetragonal construction to elements in $\cT^o_6$. 

\begin{defn}
Let $\cD'$ be the moduli space of pairs $(C,M)$ with $C\in \cM_6$ and $M\in \Pic ^4 C$, such that $h^0(C,M)=2$ and there exist three points $x_1,x_2,x_3 \in C$ satisfying $\omega_C\otimes M^{-1}\cong\cO_C(2x_1+2x_2+2x_3)$. This is a locus in the moduli space of genus $6$ curves with a marked $g^1_4$; we define $\cD \subset \cM_6$ as the closure of the image of the forgetful map restricted to $\cD'$.
\end{defn}

We have the following inclusion:

\begin{lem} \label{prop: trig locus}The trigonal locus is contained in $\cD$.
\end{lem}
\begin{proof}
Let $T\in \cM_6$ be a general trigonal curve, and let $N$ denote the $g^1_3$ on $T$.
On the one hand, $h^0(T,\omega_T\otimes N^{-2})=h^0(T,N^{ 2})-1$ by Riemann-Roch.
On the other hand, since $T$ is not hyperelliptic we have $h^0(T, N^2)<4$, whereas the natural inclusion $\Sym^2H^0(T,N)\subset H^0(T,N^2)$ yields $h^0(T,N^ 2)\ge 3$. It follows that $\vert \omega_T\otimes N^{-2}\vert$ is a $g^1_4$ on $T$, which proves the result.
\end{proof}

Observe that, geometrically, a non-trigonal $(C,M)$ lies in $\cD'$ if the image of the curve by the morphism attached to $\omega_C\otimes M^{-1}$ (a plane sextic with four nodes) has a tritangent line.
There are four additional tetragonal series on $C$, by considering the pencils of lines through each of the nodes. 
Recall that generically, a curve of genus $6$ has five $g^1_4$. We fix one of the nodes $p$ and denote by $M_p$ the line bundle attached to the pencil.
Since the canonical series corresponds to the cubics through the four nodes, the linear series $\vert \omega_C \otimes M_p^{-1}\vert $ is cut out by conics through the other three nodes. Then $(C,M_p)\in \cD'$ if, and only if, one of these conics is tangent to $C$ at three points.
The next proposition asserts that this is not possible generically. 

\begin{prop}\label{prop:gen1Tritang}
For a general plane sextic with four nodes and one tritangent line (at three smooth points of the sextic), there are no conics passing through three of the nodes and tangent to the sextic in three smooth points. In other words, the forgetful map $\cD' \to \cD$ has degree $1$.
\end{prop}

\begin{proof}
Since the existence of a tritangent line and a tritangent conic (passing through three of the four nodes) are closed properties, it is enough to exhibit a sextic with (at least) four nodes and one tritangent line for which such a conic does not exist.

We consider a smooth plane quartic $D\subset \bP^2$ with a fixed bitangent line $l$. Set $P\in l\setminus (l\cap D)$ and choose two different lines $r_1, r_2$, with $P\in r_1 \cap r_2$ and $l\neq r_1, r_2$.
Our sextic is $S=D\cup r_1 \cup r_2$, and the four nodes are two points in $D\cap r_1$ and two points in $D\cap r_2$.
The line $l$ is, by construction, tangent to $S$ at three points different from the nodes.
A conic through three of the nodes will be tritangent to $S$ (at three points different from the nodes) if passes through some other point in $r_i\cap D$ and it is bitangent to $D$.
We want to see that for a general choice of $P$ and $r_1, r_2$, this conic does not exist for $S$. 

To this end, we define two natural sets in the symmetric product $D^{(4)}$ of the quartic. First we consider $I_l^0$ to be the closure of
\[
\set{(x_1+x_2,x_3+x_4)\in D^{(2)}
\times D^{(2)}
 \mid x_1x_2, x_3x_4 \text{ are different lines and }x_1x_2 \cap x_3x_4\in l},
\]
and let $I_l$ be the image of $I_l^0$ under the addition map $D^{(2)}\times D^{(2)}\to D^{(4)}$.
This closed set has dimension $3$, and parametrizes sets of four nodes in a sextic of the form $D\cup r_1 \cup r_2$.

The second set $W$, is defined as the preimage of the surface $2D^{(2)}\subset \Pic^4 D$ by the surjective map (with $1$-dimensional general fiber)
\[
\psi: D^{(4)} \to \Pic^4 D; \quad E\mapsto \cO_D(2H-E),
\]
where $H$ is the divisor obtained by intersecting $D$ with a line.
Notice that $W$ parametrizes sets of $4$ points in $D$ such that there is a bitangent conic through the points. 

To conclude the proof, we only need to show that $I_l$ and $W$ are different threefolds inside $D^{(4)}$. 
More precisely, it is enough to see that $I_l$ is not contained in $W$. This is an immediate consequence of the following claim. 

\begin{claim} The map $\psi$ restricted to $I_l$ is dominant.
\end{claim}
Indeed, let $E=p_1 + \ldots + p_4 \in D^{(4)}$ a general effective divisor of degree $4$ in the quartic $D$.
That is, $\cO_D(E)$ is a general element in $\Pic^4 D$. 
Let $f_E:D\to \bP^1=\bP H^0(D,2H-E)^*$ the map sending $x\in D$ to the conic through $p_1, p_2, p_3, p_4 $ and $x$, and let $\Gamma_E\subset D^{(2)}$ be the closure of the set of pairs $x+x'$ with $ f_E(x)=f_E(x')$. 

There is a natural involution $\gamma_E $ in $\Gamma_E$ and a natural $\Gamma_E\to \bP^1$ of degree $6$.
We denote by $T_E$ the quotient $\Gamma _E/\gamma_E$.
Note that $T_E$ is trigonal and parametrizes the vertices of diagonal triangles of sets of vertices $V_1, V_2, V_3, V_4$ such that $p_1, p_2, p_3, p_4, V_1, V_2, V_3, V_4$ lie on a conic.
In particular, $T_E$ can be seen as a curve in the same projective plane.
Then, consider a point in $T_E\cap l$.
By construction, there is an element in $I_l$ mapping to $E$.
Hence, $\psi_{\vert I_l}$ is dominant. 
\end{proof}

\begin{rem}
  It will follow from \autoref{mainthm} that $\cD$ is an irreducible divisor of $\cM_6$. 
\end{rem}

In the second part of this section, we consider trigonal curves of genus $6$ as curves of bidegree $(3,4)$ in $\bP^1\times\bP^1$.
Indeed, for a trigonal curve $(T,N)$ of genus $6$, $\omega_T\otimes N^{-2}$ is a $g_4^1$ and we can consider the map $T\to \bP^1\times\bP^1$ provided by $\omega_T\otimes N^{-2}$ and $N$ in each factor.

Let $\cC^{(3,4)}$ be the set of all smooth $(3,4)$-curves in $\bP^1\times\bP^1$, modulo the action of the linear transformations in each projective line.
Observe that $\dim \cC^{(3,4)}=\dim \cM_6^{trig}=13$, and there is a birational map $\cM_6^{trig} \dasharrow \cC^{(3,4)}$.

For later use we need to study those curves which are bitangent to a vertical line. We define $\cC_b^{(3,4)}\subset\cC^{(3,4)}$ as the closure of the set of classes of curves $C \in |\cO_{\bP^1\times\bP^1}(3,4)|$ for which the first projection $C\to \bP^1$ has a fiber of the form $2x+2y$.
The general element of the preimage in $\cM_6^{trig}$ is a curve $T$ whose unique $g^1_3$, $N$, satisfies that $\omega_T\otimes N^{-2}$ has a divisor of this form.

\begin{prop}\label{prop:transversality}
  The set $\cC_b^{(3,4)}$ (and its preimage in $\cM_6^{trig}$) is a generically reduced divisor.
\end{prop}
\begin{proof}
We use coordinates $([x:y],[u:v])$ in $\bP^1\times \bP^1$. We consider the following $1$-dimensional family of $(3,4)$-curves
$\cV=\set{P_\alpha(x,y,u,v)=0\mid \alpha\in \bC}\subset |\cO_{\bP^1\times\bP^1}(3,4)|$, where:
\[
P_\alpha(x,y,u,v)=(x^3+y^3)u^4-2x^3u^3v+(1-\alpha)x^3u^2v^2+2\alpha x^3uv^3+(-\alpha x^3+x^2y+y^3)v^4.
\]
By abuse of notation, we still denote by $P_{\alpha}$ the curve with equation $P_{\alpha}(x,y,u,v) =0$.
Notice that the class of $P_0$ determines a general point inside $\cC_b^{(3,4)}$ since it is smooth and $P_0(1,0,u,v)$ has two double roots.

It is a well known fact from polynomial algebra 
that an equation of the form $Au^4+Bu^3v+Cu^2v^2+Duv^3+Ev^4=0$ is biquadratic if, and only if, the following two equations (in terms of the coefficients) hold:
 \[
\begin{aligned}
 \Delta(A,B,C,D,E)&= 256 A^3E^3 - 192 A^2BDE^2 -128 A^2C^2E^2 + 144 A^2CD^2E -27 A^2D^4 \\
          &\phantom{=} + 144 AB^2CE^2 - 6 AB^2D^2E - 80 ABC^2DE + 18 ABCD^3 + 16 AC^4E \\
          &\phantom{=} - 4 AC^3D^2 - 27 B^4C^2 + 18 B^3CDE - 4 B^3D^3 - 4 B^2C^3E +B^2C^2D^2=0,\\
  d(A,B,C,D,E)& =64A^3E-16A^2C^2+16AB^2C-16A^2BD-3B^4=0.
\end{aligned} 
 \]
 By replacing $A=x^3+y^3$, $B=-2x^3$, $C=(1-\alpha)x^3$, $D=2\alpha x^3$, $E= -\alpha x^3+x^2y+y^3$, we obtain the family of equations $\Delta_{\alpha}(x,y)=0, d_{\alpha}(x,y)=0$, depending on the parameter $\alpha $.
 By using resultants it is not hard to see that there is no $\alpha \neq 0$ and near $0$, such that $\Delta_{\alpha}$ and $d_{\alpha}$ share a root. 
 Therefore, for these $\alpha $, $P_{\alpha}$ has not a vertical bitangent line, in particular $P_0(x,y,u,v)$ and $P_{\alpha}(x,y,u,v)$ are not projectively equivalent.
 Thus, $\cV$ defines an actual curve around the class of $P_0$ in $\cC^{(3,4)}$. 

Consider the $\bP^1$-bundle
\[
\cW=\set{(P_{\alpha},D_{\alpha}) \mid D_{\alpha}\in g_4^1=\vert K_{p_{\alpha}}-2g^1_3\vert }\to \cV.
\]
Around $P_0\in \cV$ we have a section of this $\bP^1$-bundle given by the fiber at the point $[x:y]=[1:0]\in\bP^1$, which is given by the zeroes of the polynomial: 
\[
(u-v)^2(u^2-\alpha v^2).
\]
On this section we always have a double root, hence the intersection with the preimage of $\cC_b^{(3,4)}$ corresponds to the locus where the corresponding homogeneous degree $4$ equation has a second double root. Due to previous discussion we have to impose the equation $d_{\alpha}(x,y)=0$, which in our case is: 
\[
-16\alpha^2 - 32\alpha.
\]
Since the linear term of this polynomial in $\alpha$ is not zero, we conclude that the locus $\cC_b^{(3,4)}$ is reduced.
\end{proof}

\begin{rem} \label{althypersurface}
As suggested by Jieao Song, one could prove alternatively this proposition by checking that the hypersurface in $|\cO_{\bP^1\times\bP^1}(3,4)|$ of curves with a vertical bitangent line is integral of degree 24. Indeed, this hypersurface is dominated by the incidence variety
 \[
 \set{([a_0:a_1],F)\in\bP^1\times |\cO_{\bP^1\times\bP^1}(3,4)|: \text{ $F(a_0,a_1,u,v)$ is a perfect square in $u,v$}},
 \]
 which has a structure of $\bP^{15}$-bundle over $\bP^1\times\bP^2$ (here $\bP^2$ is parametrizing perfect square polynomials of degree 4 in $u,v$); hence both, the incidence and the hypersurface, are irreducible.
 A Chern class computation says that the degree of the hypersurface is $24$; then, by picking a random pencil in $|\cO_{\bP^1\times\bP^1}(3,4)|$ with Macaulay2, one finds 24 distinct intersection points, which proves that the hypersurface is also reduced.
\end{rem}

\section{The tetragonal construction on $\cT_6^o$}\label{sec:Deg1overImage}

Given the structure of the Prym map in genus 6 (see \autoref{prymmapgenus6}), in order to describe $\cP_6^{-1}(\cP_6(\cT_6^o))$ it becomes essential to understand the tetragonal construction applied to a general element in $\cT_6^o$. This was partially done in \cite[Proposition~7.2]{lnr}, where it was proved:

\begin{prop}
  Let $(C_i,\eta_i,M_i)$ ($i=1,2,3$) be a tetragonally related triple of smooth Prym curves $(C_i,\eta_i)\in\cR_6$ with a $g^1_4$ $M_i$ on $C_i$. If $(C_1,\eta_1)\in\cT^o_6$ is general, then $JC_2,JC_3\in\cP_7(\cT^o_7)$.
\end{prop}

More precisely the pairs $(C_2,M_2),(C_3,M_3)$ correspond, under Recillas' trigonal construction, to two covers $(T,\mu_2),(T,\mu_3)\in\cT^o_7$ of the same trigonal curve $T$ of genus 7. It is thus important to understand Prym varieties of covers of trigonal curves endowed with an odd semicanonical pencil. This is the content of the next proposition:

\begin{prop}\label{trigonal_constr_and_v2}
  Let $(T,\mu)$ be a trigonal, non-hyperelliptic Prym curve of genus $g\geq6$, and let $C$ be the curve of genus $g-1$ (endowed with a $g^1_4$ $M$) obtained by trigonal construction. If $(T,\mu)\in\cT^o_g$, then
  \[
  \omega_C\otimes M^{-1}\cong \cO_C\left(2(x_1+\ldots+x_{g-4})\right)
  \]
  for some $x_1,\ldots,x_{g-4}\in C$.
\end{prop}
\begin{proof}
  We will exploit the fact that (a suitable translate of) the Prym-Brill-Noether locus $V^2(T,\mu)$ is symmetric with a singular 2-torsion point, since $(T,\mu)\in\cT^o_g$. We have
  \[
  V^2(T,\mu)\cong W_{g-4}(C)\cup \left(-W_{g-4}(C)\right)
  \]
  (see \cite[Section~4.C]{ho}). More precisely, recall that the double cover $\tT\to T$ can be canonically embedded in $\Pic^2 C$ as 
\[
\tT=\set{p+q\in C^{(2)}\mid h^0(L(-p-q))>0}\subset C^{(2)}\to \Pic^2 C.
\]
Thus, for a fixed $p_0+q_0\in\tT$, we can consider the (non-canonical) Abel-Prym curve
\[
\tT\hookrightarrow \Pic^0 C,\qquad p+q\mapsto\cO_C(p+q-p_0-q_0).
\]
Using this embedding, we have a model of $V^2(T,\mu)$ in $\Pic^{g-2}C$ as the theta-dual of $\tT$:
\[
V^2(T,\mu)=\set{\xi\in\Pic^{g-2}C\mid \tT+\xi\subset \Theta_C}
\]
(where $\Theta_C\subset\Pic^{g-2}C$ denotes the Riemann theta divisor), which consists of two components:

\begin{itemize}
 \item $W_{g-4}(C)+p_0+q_0$.
 \item $K_C-r_0-s_0-W_{g-4}(C)$, where $r_0,s_0\in C$ satisfy $M=\cO_C(p_0+q_0+r_0+s_0)$.
\end{itemize}

Note that if $C$ is non-hyperelliptic, then $W_{g-4}(C)$ does not admit a symmetric translate. 
Indeed, suppose that $W_{g-4}(C)-\alpha=-W_{g-4}(C)+\alpha$, that is, there exists $\alpha$ such that for all $x_1+\ldots+x_{g-4}$ there exists $y_1+\ldots+y_{g-4}$ such that $2\alpha=x_1+\ldots+x_{g-4}+y_1+\ldots+y_{g-4}$.
Then $h^0(C,2\alpha)\geq g-4+1$, and since $\deg \alpha =2g-8$, this implies that the Clifford index of $C$ is 0, hence hyperelliptic.

Therefore, $W_{g-4}(C)$ does not admit symmetric translates and $V^2(T,\mu)\subset \Pic^{g-2}C$ becomes symmetric exactly when its two components are exchanged by the involution $L\mapsto \omega_C\otimes L^{-1}$ on $\Pic^{g-2}C$ (recall that in the torsor $\Pic^{g-2}C$ of $JC$, the inversion map $-1$ corresponds to taking the adjoint line bundle).
This is equivalent to translate by $\alpha\in\Pic^0 C$ such that
\[
W_{g-4}(C)+p_0+q_0+\alpha=W_{g-4}(C)+r_0+s_0-\alpha,
\]
namely $\alpha^{\otimes 2}=\cO_C(r_0+s_0-p_0-q_0)$.

Note that any theta-characteristic on $C$ (which now plays the role of a 2-torsion point in $\Pic^{g-2}C$) lying in this symmetric model of $V^2(T,\mu)$ is automatically a singular point of $V^2(T,\mu)$ (since it lies in both components).
The condition of $W_{g-4}(C)+p_0+q_0+\alpha$ containing a theta-characteristic is easily seen to be equivalent to the existence of $x_1,\ldots,x_{g-4}\in C$ such that $2(x_1+\ldots+x_{g-4})\in|\omega_C\otimes M^{-1}|$, which finishes the proof.
\end{proof}

By combining the two propositions above, we obtain the following relation between $\cT^o_6$ and the locus $\cD$ introduced in \autoref{sec:divisorD}:

\begin{cor}\label{prop: how tetr works in T_odd}
Let $(C_i,\eta_i,M_i)$ ($i=1,2,3$) be a tetragonally related triple of smooth Prym curves $(C_i,\eta_i)\in\cR_6$ with a $g^1_4$ $M_i$ on $C_i$. If $(C_1,\eta_1)\in\cT^o_6$ is general, then the pairs $(C_2,M_2)$ and $(C_3,M_3)$ belong to $\cD'$.
\end{cor}

An immediate consequence of this discussion is:

\begin{cor}\label{cor:Ddivisor}
The locus $\cD\subset \cM_6$ is a divisor.
\end{cor}

Conversely, a general element in $\pi^{-1}\cD$ (with no condition on the 2-torsion line bundle!) is tetragonally related to a Prym curve in $\cT^o_6$:

\begin{prop}\label{prop: how tetr works in D}
Let $(C_i,\eta_i,M_i)$ ($i=1,2,3$) be a tetragonally related triple of smooth Prym curves $(C_i,\eta_i)\in\cR_6$ with a $g^1_4$ $M_i$ on $C_i$.
If $(C_2,M_2)\in\cD'$ is general, then we have $(C_1,\eta_1)\in\cT^o_6$ and $(C_3,M_3)\in\cD'$.
\end{prop}
\begin{proof}
Let $T\in \cM_7$ denote the trigonal curve and $\{0,\mu_1,\mu_2,\mu_3\}\subset JT[2]$ the totally isotropic subgroup associated to the tetragonally related triple by \autoref{trigtetr}, so that $P(T,\mu_i)\cong JC_i$ under Recillas' construction.

By arguing as in the last two paragraphs of the proof of \autoref{trigonal_constr_and_v2}, one sees that the assumption $(C_2,M_2)\in\cD'$ is equivalent to $V^2(T,\mu_2)$ having a symmetric translate with a singular 2-torsion point. We claim that $(T,\mu_2)\in\cT^o_7$.

Indeed, let $f:\tT_2\to T$ be the étale double cover corresponding to $(T,\mu_2)$. When $V^2(T,\mu_2)$ is translated from $P^-(T,\mu_2)$ to the actual Prym $P(T,\mu_2)$, the symmetric models for $V^2(T,\mu_2)$ are obtained with translation by a theta-characteristic of $\tT$ lying in $P^-(T,\mu_2)$.
Such a theta-characteristic is of the form $f^*L_T$, for some theta-characteristic $L_T$ on $T$.
Since $f^*L_T\in P^-(T,\mu_2)$, we have $h^0(L_T)+h^0(L_T\otimes\mu_2)=h^0(f^*L_T)\geq 3$ is odd. Hence (after possibly replacing $L_T$ by $L_T\otimes\mu_2$), we have one of the following:

\begin{itemize}
    \item Either $L_T$ is a semicanonical pencil such that $L_T\otimes\mu_2$ is odd, which implies $(T,\mu_2)\in\cT^o_7$.

    \item Or $L_T$ is an odd theta-characteristic with $h^0(L_T)\geq3$. This implies that $T$ has Maroni invariant 3 (see \cite[Proposition~3.2]{tctrigonal}), which is impossible if $(C_2,M_2)$ is a general element of $\cD'$, since the locus of trigonal genus 7 curves with Maroni invariant 3 has dimension 13.
\end{itemize}

Now by the Riemann-Mumford relation (see \cite{mu3} or \cite[Theorem~1.13]{harris}), we have
\[
h^0(T,L_T)+h^0(T,L_T\otimes\mu_1)+h^0(T,L_T\otimes\mu_2)+h^0(T,L_T\otimes\mu_3)\equiv 0\pmod{2}
\]
(here we use that $\mu_3=\mu_1\otimes\mu_2$), which implies that $L_T\otimes\mu_1$ and $L_T\otimes \mu_3$ are theta-characteristics of distinct parities.

Say $(T,\mu_1)\in\cT^e_7$ and $(T,\mu_3)\in\cT^o_7$. Then on the one hand, the pair $(C_3,M_3)$ lies in $\cD'$ by \autoref{trigonal_constr_and_v2} and, on the other hand, $JC_1=P(T,\mu_1)\in\theta_{null}\subset\cA_6$, namely $C_1\in\cT_6$. Moreover, since $\cT^e_6$ is closed under the tetragonal construction (see \cite[Proposition~7.2]{lnr}), it follows that $(C_1,\eta_1)\in\cT^o_6$, which finishes the proof.
\end{proof}

Since the general $C\in\cD$ has a unique $g^1_4$ whose adjoint defines a plane with a tritangent line (see \autoref{prop:gen1Tritang}), we obtain the following part of \autoref{mainthm}:

\begin{thm}\label{thm:genInjonTo}
Let $\cZ\subset\cA_5$ the divisor obtained as the closure of $\cP_6(\cT^o_6)$. Then:
\begin{enumerate}
  \item\label{item1:genInjonTo} The restricted Prym map $\cP_6|_{\cT_6^o}$ is generically injective.
  \item\label{item2:genInjonTo} $\pi^{-1}\cD\subset\cP_6^{-1}(\cZ)$, and $\deg\left(\cP_6|_{\pi^{-1}\cD}\right)=10$.
\end{enumerate}
\end{thm}
\begin{proof}
 Recall that, in the general fiber of $\cP_6$, the tetragonal relation equals the incidence relation on the 27 lines of a cubic surface. 
 
 By \autoref{prop: how tetr works in T_odd}, starting with an element of $(C_1,\eta_1)\in\cT^o_6$, we obtain 10 elements of the fiber which lie in $\pi^{-1}\cD$ (by tetragonal construction with \emph{any} of the five $g^1_4$'s of $C_1$). 
 To determine the rest of the fiber, we must fix \emph{any} pair $(C_2,\eta_2),(C_3,\eta_3)$ tetragonally related to $(C_1,\eta_1)$, and apply the tetragonal construction to $(C_2,\eta_2)$ and $(C_3,\eta_3)$ by using the rest of the $g^1_4$'s on $C_2$ and $C_3$. 
 By \autoref{prop:gen1Tritang}, $C_2$ and $C_3$ have no more ``tritangent $g^1_4$'s", hence the tetragonal construction can't bring us back to $\cT^o_6$ (again by \autoref{prop: how tetr works in T_odd}). Therefore $(C_1,\eta_1)$ is the unique element of the fiber which lies on $\cT^o_6$, which proves \eqref{item1:genInjonTo}.
 
 In order to prove \eqref{item2:genInjonTo}, observe that for any $C\in\cD$ and $\eta\in JC_2\setminus\{\cO_C\}$, $(C,\eta)$ is tetragonally related to an element in $\cT_6^o$ (by \autoref{prop: how tetr works in D}). This implies $\pi^{-1}\cD\subset\cP_6^{-1}(\cZ)$. Furthermore, $\deg\left(\cP_6|_{\pi^{-1}\cD}\right)=10$, since \autoref{prop: how tetr works in T_odd} and \autoref{prop: how tetr works in D} show that elements of $\pi^{-1}\cD$ in the fiber are exactly those tetragonally related to the unique element of $\cT_6^o$.
\end{proof}

Also note that from \autoref{prop: trig locus} and \autoref{thm:genInjonTo}.\eqref{item2:genInjonTo} we directly obtain that
\begin{cor}\label{cor:ScontainsJ}
The Jacobian locus is contained inside $\cZ$.
\end{cor}

\section{Monodromy argument}\label{sec:monodromy}
In this section we compute the monodromy group of $\cP_6^{-1}(\cZ)\to\cZ$. 
Note that \autoref{thm:genInjonTo}.\eqref{item1:genInjonTo} already shows that, as a subgroup of the monodromy $WE_6$ of the entire Prym map $\cP_6$, it is contained in the stabilizer $WD_5$ of a fixed line. 

We start by recalling the geometric objects in the fiber of $\cP_6$ over the Jacobian locus of $\cA_5$, as they are introduced in \cite{donsmith}. Then we adapt part of Donagi's strategy (\cite[Section 4]{do_fibres}) to our context.

Let $X$ be a general smooth curve of genus $5$. The canonical map embeds $X$ in $\bP^4=\bP H^0(X,\omega_X)^\vee$ as the complete intersection of three quadrics; we will denote by $I_X(2)$ the vector space of the equations of the quadrics containing $X$.
The singular quadrics in $|I_X(2)|=\bP^2$ define a smooth plane quintic $\Gamma$, which admits the Brill-Noether curve $W^1_4(X)\subset JX$ as an étale double cover; over a singular quadric cone $Q\in\Gamma$, one has two $g^1_4$'s (adjoint to each other) swept out by the two families of $2$-planes on $Q$.

By Recillas' construction, there is a trigonal $(T,\mu)\in \cR_6$ attached to any $M\in W^1_4(X)$. The isomorphism $P(T,\mu)\cong JX$ provides a natural identification $H^0(T,\omega_T\otimes\mu)\cong H^0(X,\omega_X)$ of their tangent spaces at the origin. By considering the Prym-canonical embedding of $T$ in
  \[
 \bP H^0(T,\omega_T\otimes\mu)^\vee=\bP^4=\bP H^0(X,\omega_X)^\vee,
  \]
quadrics of $|I_X(2)|$ cut on $T$ its unique basepoint-free $g^1_4$ (namely $K_T - 2 g^1_3$) plus a base locus which is the ramification locus $R$ of the $g^1_3$. Furthermore, the singular quadric $Q_M\in \Gamma$ associated to $M$ is the unique quadric containing both $X$ and $T$  (\cite[Part~III]{donsmith}).

Now let $\cR_6^{trig}\subseteq \cR_6'$ denote the closure in $\cR_6'$ of the locus of double covers of smooth trigonal curves, and let $\cJ\subset \cA_5$ denote the Jacobian locus.  
Let $JX$ be a general element in $\cJ$, and let $(T,\mu)\in \cR_6^{trig}$ be a smooth element in $\cP_6^{-1}(JX)$.
If $N$ denotes the $g^1_3$ on $T$, then we have canonical identifications: 
\[
\begin{aligned}
    & \cN^*_{\cJ/\cA_5,JX}\cong I_X(2), \\
    & \cN^*_{\cR_6^{trig}/\cR_6,(T,\mu)}\cong  H^0\big(T,\omega_T^2(-R)\big)\cong H^0\big(T,\omega_T\otimes N^{-2}\big)
\end{aligned}
\]
(see \cite{griff} and  \cite[\S III.4 \& Appendix]{donsmith}, respectively). 

\begin{prop}[{\cite[Corollary~4.3]{donsmith}}]\label{prop:normals_and_projection} 
    The projectivization of the codifferential of $\cP_6$ \begin{equation}\label{eqn:conormal}
  d \cP_6^*\colon \cN^*_{\cJ/\cA_5,JX}\to \cN^*_{\cR_6^{trig}/\cR_6,(T,\mu)}
\end{equation}
sends a quadric of $|I_X(2)|$ to its intersection with $T$ (minus the ramification $R$).
Namely, it is the projection from the quadric $Q_M\in\Gamma$ (containing both $X$ and $T$) to a supplementary line.
\end{prop}

\begin{rem}\label{delPezzoquadric}
It follows that the points of $\bP\big(\cN^*_{\cR_6^{trig}/\cR_6,(T,\mu)}\big)\cong\bP\big(\cN_{\cR_6^{trig}/\cR_6,(T,\mu)}\big)$ are in correspondence with degree $4$ del Pezzo surfaces containing $X$ and contained in $Q_{M}$.
\end{rem}

\vspace{3mm}

In \cite[Section 4]{do_fibres}, Donagi considers a blown up map $\tcP_6'$, obtained by blowing up $\cR_6'$ and $\cA_5$ along several geometric loci. This enables him to prove \autoref{prymmapgenus6}.\eqref{mongenus6}. The map $\tcP_6'$ is generically finite over the strict transform $\tcJ$ of the Jacobian locus, and he also proves that $WD_5$ is the monodromy of $\tcP_6'$ over $\tcJ$. 
For the convenience of the reader, we sketch the argument for the second assertion. The picture is the following:
\begin{equation}\label{donagiblowup}
\xymatrix{
\tcR_6^{trig}\subset \Bl_{\cR^{trig}_6\cup\cR\mathcal{Q}^+}\cR_6'\ar[d]\ar[rr]^(.6){\tcP_6'}&& \Bl_{\cJ}\cA_5 \supset \tcJ \ar[d]\\
\cR_6^{trig}\subset \cR_6'\ar[rr]_{\cP_6'}&&\cA_5\supset \cJ
}    
\end{equation}

Here $\cR\mathcal{Q}^+$ is a component of the locus of double covers of plane quintics, and $\tcP_6'$ is the blown up map from the blow up of $\cR_6'$ along $\cR_6^{trig}\cup \cR\mathcal{Q}^+$ to the blow up of $\cA_5$ along the Jacobian locus $\cJ$.
We denote by $\widetilde{\cR\mathcal{Q}}^+$, $\tcR_6^{trig}$ and $\tcJ$ the corresponding exceptional divisors. Then:

\begin{itemize}
    \item The fiber of $\tcJ$ over a general Jacobian $JX\in\cJ$ is $\bP\big(\cN_{\cJ/\cA_5,JX}\big)\cong |I_X(2)|^\vee$, namely the set of degree 4 del Pezzo surfaces containing $X$.

    \item The fiber of $\tcR_6^{trig}$ over a general $(T,\mu)\in\cR_6^{trig}$ is $\bP\big(\cN_{\cR_6^{trig}/\cR_6,(T,\mu)}\big)$. By \autoref{delPezzoquadric} this is the set of degree 4 del Pezzo surfaces containing $X$ and contained in the singular quadric $Q_M$, where $M\in W^1_4(X)$ is obtained from $(T,\mu)$ by Recillas' construction.
\end{itemize}

If $\cP_6(T,\mu)=JX$, then the induced map at the level of exceptional divisors is clear in terms of this description.

Let $(X,S)$ be a general element in $\tcJ$, namely $X\in\cM_5$ is general and $S$ is a degree 4 del Pezzo surface containing the canonical model of $X$. The preimage under $\tcP_6'$ of $(X,S)$ consists of 27 points: one element in $\widetilde{\cR\mathcal{Q}}^+$, its 10 tetragonally related elements (all of them in $\tcR_6^{trig}$) and 16 Wirtinger covers (corresponding to the 16 lines in $S$).
Therefore, the monodromy group of $\tcP_6'$ over $\tcJ$, as a subgroup of the entire monodromy $WE_6$, is contained in $WD_5$. 

In order to conclude that it equals $WD_5$, one observes that by fixing a general curve $X\in \cM_5$ and considering all degree 4 del Pezzo surfaces containing $X$, one obtains all (isomorphism classes of) degree $4$ del Pezzo surfaces. Since the monodromy group of the universal family of 16 lines on smooth degree 4 del Pezzo surfaces is isomorphic to $WD_5$, the assertion follows.

Now we adapt these ideas to study the monodromy of the Prym map restricted to the preimage of $\cZ$. We prove the following theorem which is part of \autoref{mainthm}. 

\begin{thm}\label{thm:monodromyOverZ} The monodromy of $\cP_6\colon \cP_6^{-1}\cZ\to \cZ$ is $WD_5$.
\end{thm}

Since $\cZ$ contains the Jacobian locus $\cJ$ by \autoref{cor:ScontainsJ}, we can consider $\hcZ=\Bl_\cJ\cZ$ which is the strict transform of $\cZ$ in $\Bl_\cJ\cA_5$.
If we want to study the monodromy over $\cZ$, it is enough to consider paths inside $\hcZ\cap \tcJ$ (the exceptional divisor of $\Bl_\cJ\cZ=\hcZ$).

To this end, it is absolutely essential to understand the normal bundle $\cN_{\cJ/\cZ}$ at a general point of $\cJ$.
Naively, we can consider the diagram
\begin{equation*}\label{blowup}
\xymatrix{
\widetilde{\cR_6^{trig}\cap \cT_6^o} \subset \Bl_{\cR_6^{trig}\cap \cT_6^o}\cT_6^o \ar[d]\ar[r]&\Bl_\cJ\cZ\supset \tcJ\cap \hcZ \ar[d]\\
	\cR_6^{trig}\cap \cT_6^o\subset \cT_6^o\ar[r]_{\cP'_6} & \cZ\supset \cJ.
}
\end{equation*}
(we warn the reader that this diagram is not the analogue of \eqref{donagiblowup}, where the entire blown up map is depicted). Since $\cP_6'$ defines a birational map between $\cT^o_6$ and $\cZ$, we can identify the normal bundle $\cN_{\cJ/\cZ}$ at a general point of $\cJ$ with the normal bundle $\cN_{\cR_6^{trig}\cap \cT_6^o/\cT_6^o}$. 

Before going to the proof of \autoref{thm:monodromyOverZ}, we need two results about the intersection $\cR_6^{trig}\cap \cT_6^o$.

\begin{lem}\label{lem:trans}
  The intersection $\cR_6^{trig}\cap \cT_6^o$ is generically transverse.
\end{lem}
\begin{proof}
Since $\cR_6\to \cM_6$ is \'etale, it is enough to prove that the intersection $\cM_6^{trig}\cap \cT_6$ is transverse. 
We claim that the general point in $\cM_6^{trig}\cap \cT_6$ consists of a trigonal $T\in \cM_6$ such that its unique $g^1_3$ $N$ satisfies that
$N(x+y)$ is a semicanonical pencil, for some $x,y\in T$.

Indeed, let $T\in\cM_6^{trig}$ be non-hyperelliptic and $N\in W_3^1(T)$, and suppose that $L$ is a semicanonical pencil on $T$. 
We want to prove that $L\otimes N^{-1}$ is effective. This is equivalent to $h^0(T,L\otimes N)\geq 4$, since 
$h^0(T,L\otimes N^{-1})=h^0(T,L\otimes N)-3$ by Riemann-Roch and Serre duality.

If the multiplication map
\[
\mu_{L,N}:H^0(L)\otimes H^0(N)\lra H^0(L\otimes N)
\]
is injective, then $h^0(L\otimes N)\geq 4$. Otherwise, by the basepoint-free pencil trick $0\neq\ker(\mu_{L,N})=H^0(N\otimes L^{-1}(B))$, where $B$ is the base locus of $L$. In particular, $B$ has degree $2$ and $L(-B)$ defines the $g^1_3$ $N$, which also proves the claim.

Our claim implies that the general point in $\cM_6^{trig}\cap \cT_6$ can be interpreted as a curve of bidegree $(3,4)$ in $\bP^1\times \bP^1$, whose $g^1_4$ admits a fibre of the form $2x+2y$. In other words, it gives a point of $\cC_b(3,4)$.
Thus, the transversality follows from \autoref{prop:transversality}.
\end{proof}

We also want to prove that the intersection $\cR_6^{trig}\cap \cT_6^o$ is still dominant over $\cJ$.

\begin{rem}\label{rem:trig_constr_T6o}
Let us observe that combining Recillas' birational map 
together with \autoref{trigonal_constr_and_v2}, we have that giving a general element of $\cR_6^{trig}\cap \cT_6^o$ is equivalent to giving a general pair $(X, M)$ such that $X\in \cM_5$, $M\in W^1_4(X)$ and $\omega_X=M(2p+2q)$ for some $p,q\in X$.
\end{rem}

\begin{prop}\label{prop:dominant}
  The map $\cP'_6\colon \cR_6^{trig}\cap \cT_6^o\to \cJ$ is generically finite, hence dominant. 
\end{prop}
\begin{proof}
  According to \autoref{rem:trig_constr_T6o}, we need to show that the forgetful map 
  \[\set{(X, M) \mid
  X\in \cM_5, \;M\in W^1_4(X), \;\omega_X=M(2p+2q) \text{ for some } p, q }\lra \cM_5\]
  is dominant (or equivalently, generically finite). 

  This follows from a cohomological computation. Indeed, let $C^1_4\subset X^{(4)}$ be the subvariety parametrizing effective divisors of degree 4 on $X$ which move in a linear series of dimension at least 1. By \cite[Theorem p.~326]{acgh}, its fundamental class in $N^1(X^{(4)})$ is
  \[c^1_4=\frac{1}{2}(\theta^2-2x\theta)\]
  where, as usual, $x$ is the class of $p+X^{(3)}$ for some point $p$, and $\theta$ is the class of the theta divisor of $JX$ (pulled back to $X^{(4)}$ via the Abel-Jacobi map).

  On the other hand, let $\Delta_2\subset X^{(4)}$ denote the image of the diagonal map 
  \[
  X^2\to X^{(4)}, \ p+q\mapsto 2p+2q.
  \]
  By \cite[Chap.~VIII, Prop.~5.1]{acgh}, we have that $\Delta_2= 4(32x^2+\theta^2-10x\theta)$.
  Therefore, it follows from \cite[Lemma 1]{kouvi} that \[c^1_4\cdot \Delta_2= 104x^2\theta^2+2\theta^4-24x\theta^3-128x^3\theta=240,\]
  which gives the finite degree we were looking for. 
\end{proof}

Now we are ready to prove the main theorem of this section.

\begin{proof}[Proof of \autoref{thm:monodromyOverZ}]
By \autoref{prop:dominant}, given $X$ a general smooth curve of genus 5,
there exists a finite number of $(T,\mu)\in \cR_6^{trig}\cap \cT_6^o$ (with $T$ smooth) such that $JX=P(T,\mu)$.

Let us choose such a $(T,\mu)\in \cR_6^{trig}\cap \cT_6^o$, and let $M\in W^1_4(X)$ be the $g^1_4$ obtained in $X$ by applying the trigonal construction to $(T,\mu)$.
First we claim that the map \eqref{eqn:conormal} factors through
\begin{equation*}
\cN^*_{\cJ/\cA_5,JX}\to \cN^*_{\cJ/\cZ,JX}\stackrel\cong\to \cN^*_{\cR_6^{trig}/\cR_6,(T,\mu)}.
\end{equation*}

The factorization is clear, thus we need to prove that the second morphism is bijective.
If $X$ is a general curve of genus $5$, it is enough to prove that $\cZ$ is not singular at $JX$. 
Since $\cZ$ is not contained in the branch divisor of $\cP_6$ and we have a birational map $\cT_6^o\to \cZ$ (by \autoref{thm:genInjonTo}), by \autoref{prop:dominant}, we get 
\[\cN^*_{\cJ/\cZ, JX}\cong \cN^*_{\cR_6^{trig}\cap \cT_6^o/ \cT_6^o, (T, \mu) }. \] 
Since the intersection $\cR_6^{trig}\cap \cT_6^o$ is generically transverse by \autoref{lem:trans}, we have 
\[
\cN^*_{\cR_6^{trig}\cap \cT_6^o/ \cT_6^o, (T, \mu) }\cong \cN^*_{\cR_6^{trig}/ \cR_6, (T, \mu) }
\]
which proves the claim. 

Now, as a consequence of the discussion after diagram \eqref{donagiblowup}, we have that the projectivized normal bundle
\[
\bP(\cN_{\cJ/\mathcal{Z},JX})\hookrightarrow \bP(\cN_{\cJ/\cA_5,JX})=\vert I_X(2)\vert ^\vee\]
is identified with the pencil of degree 4 del Pezzo surfaces $S$ containing $X$ and contained in the singular quadric $Q_{M}=Q_{\omega_X\otimes M^{-1}}\in \Gamma \subset |I_X(2)|$. Let us also recall that, by \autoref{rem:trig_constr_T6o}, the $g^1_4$ $\omega_X\otimes M^{-1}$ has a divisor of the form $2p+2q$. In other words, the quadric cone $Q_M$ contains a $2$-plane $\pi\subset Q_M$ such that $S\cap\pi$ is of the form $2p+2q$.

Finally, recall that by \autoref{thm:genInjonTo}.\eqref{item1:genInjonTo} the monodromy group of $\cP_6\colon \cP_6^{-1}\cZ\to \cZ$ is contained in $WD_5$.
Thus, if we prove that for any smooth degree 4 del Pezzo surface $S\subset \bP^4$ there exists a $2$-plane $\pi\subset \bP^4$ such that $S\cap\pi$ is of the form $2p+2q$, we will be done because, even considering paths inside $\hcZ\cap \tcJ$, we will still recover all (isomorphism classes of) degree $4$ del Pezzo surfaces, whose full symmetric group of line configuration is $WD_5$.

This holds true, since for every pair of distinct points $p, q\in S$, we can consider the 2-plane obtained as the intersection of:
\begin{itemize}
    \item The hyperplane spanned by $q$ and the tangent 2-plane to $S$ at $p$.
    \item The hyperplane spanned by $p$ and the tangent 2-plane to $S$ at $q$.
\end{itemize} 
Any quadric cone containing this $2$-plane will intersect $S$ in an appropriate genus 5 curve, which completes the proof.
\end{proof}

\begin{proof}[Proof of \autoref{mainthm}]
By \autoref{thm:monodromyOverZ} $\cP_6^{-1}(\cZ)$ consists of three irreducible components, corresponding to the three orbits of the action of $WD_5 \subset \mathfrak{S}_{27}$ on the general fiber of $\cP_6|_{\cP_6^{-1}(\cZ)}$. Furthermore, these three components dominate $\cZ$ with degrees 1, 10 and 16.

On the other hand, it follows from \autoref{thm:genInjonTo} that the irreducible component mapping with degree 1 must be $\cT^o_6$, and the irreducible component mapping with degree 10 must be $\pi^{-1}\cD$.
\end{proof}


\end{document}